\documentclass[12pt]{amsart}
\usepackage{amssymb,amscd,amsmath,enumerate,color,url,hyperref}
\usepackage[all]{xy}

\usepackage[top=30mm,right=30mm,bottom=30mm,left=30mm]{geometry}

\newtheorem{thm}{Theorem}[section]

\newtheorem{prop}[thm]{Proposition}

\newtheorem{lem}[thm]{Lemma}
\theoremstyle{remark}

\theoremstyle{definition}
\newtheorem{defn}[thm]{Definition}

\newtheorem*{thmA}{Theorem A}
\newtheorem*{thmB}{Theorem B}
\newtheorem*{corC}{Corollary C}
\newtheorem*{thmD}{Theorem D}

\numberwithin{equation}{section}

\renewcommand{\bar}{\overline}

\newcommand{\Irr}{\mathrm{Irr}}
\newcommand{\Spin}{\mathrm{Spin}}

\newcommand{\SO}{\mathrm{SO}}

\newcommand{\G}{{\mathbb{G}}}
\renewcommand{\P}{{\mathbb{P}}}
\newcommand{\F}{{\mathbb{F}}}
\newcommand{\N}{{\mathbb{N}}}

\newcommand{\uG}{{\underline{G}}}

\newcommand{\uT}{{\underline{T}}}

\newcommand{\CL}{{\mathcal{C}}}

\newcommand{\Fi}{\mathrm{Fi}}

\newcommand{\ZB}{\mathbf{Z}}
\newcommand{\CB}{\mathbf{C}}

\newcommand{\SL}{\mathrm{SL}}
\newcommand{\SU}{\mathrm{SU}}
\newcommand{\PSL}{\mathrm{PSL}}
\newcommand{\PGL}{\mathrm{PGL}}
\newcommand{\PSU}{\mathrm{PSU}}

\newcommand{\PSp}{\mathrm{PSp}}

\newcommand{\GU}{\mathrm{GU}}

\newcommand{\Sp}{\mathrm{Sp}}

\renewcommand{\sc}{\mathrm{sc}}

\newcommand{\AAA}{{\sf A}}
\newcommand{\SSS}{{\sf S}}

\newcommand{\tw}[1]{{}^#1\!}
\newcommand{\Alt}{{\raise 2pt\hbox{$\scriptstyle\bigwedge$}}}

\newcommand{\e}{\epsilon}

\def\ecn{\mathsf{ecn}}
\def\ccn{\mathsf{ccn}}
\def\iw{\mathsf{iw}}

\def\dl{\mathfrak{d}}

\begin{document}

\title{Characteristic Covering Numbers of Finite Simple Groups}

\author{Michael Larsen}
\email{mjlarsen@indiana.edu}
\address{Department of Mathematics\\
    Indiana University \\
    Bloomington, IN 47405\\
    U.S.A.}

\author{Aner Shalev}
\email{shalev@math.huji.ac.il}
\address{Einstein Institute of Mathematics\\
    Hebrew University \\
    Givat Ram, Jerusalem 91904\\
    Israel}

\author{Pham Huu Tiep}
\email{pht19@math.rutgers.edu}
\address{Department of Mathematics\\
    Rutgers University \\
    Piscataway, NJ 08854-8019 \\
    U.S.A.}

\begin{abstract}
We show that, if $w_1, \ldots , w_6$ are words which are not an identity of any (non-abelian) finite simple group,
then $w_1(G)w_2(G) \cdots w_6(G) = G$ for {\it all} (non-abelian) finite simple groups $G$. In particular, for every word $w$,
either $w(G)^6 = G$ for all finite simple groups, or $w(G)=1$ for some finite simple groups.

These theorems follow from more general results we obtain on {\it characteristic collections} of finite groups
and their covering numbers, which are of independent interest and have additional applications.
\end{abstract}

\subjclass{Primary 20D06}

\thanks{ML was partially supported by NSF grant DMS-2001349.
AS was partially supported by ISF grant 686/17 and the Vinik Chair
of mathematics which he holds. PT was partially supported by NSF grant DMS-1840702, the Joshua Barlaz Chair in Mathematics,
and the Charles Simonyi Endowment at the Institute for Advanced Study (Princeton).
The authors were also partially supported by BSF grant 2016072.}

\thanks{The authors are grateful to Frank L\"ubeck for kindly providing us with the character table of the Steinberg group
$\tw3 D_4(3)$.}
\maketitle

\section{Introduction}

The theory of word maps on groups, and on finite simple groups in particular, dates back to Borel \cite{Bo}, and has
developed significantly in the past 3 decades; see for instance \cite{MZ}, \cite{SW}, \cite{LiSh1}, \cite{La}, \cite{Sh1}, \cite{LS1},
\cite{LS2}, \cite{LBST}, \cite{LST}, \cite{GT}, as well as the monograph \cite{Se} and the survey paper \cite{Sh2}.

Most of these works are of asymptotic nature. For example, it is shown in \cite{LST} (following \cite{LS1, LS2}) that, for every
non-identity word $w$ there exists $N(w) \in \N$ such that, if $G$ is a finite simple group of order at least $N(w)$
then $w(G)^2 = G$. Recall that a {\it word} is an element of a free group $F_d$ of rank $d$, and it defines a {\it word map} $w\colon G^d \to G$
for every group $G$, induced by substitution. The image of this word map is denoted by $w(G)$.

Here and throughout this paper, by a finite simple group we mean a non-abelian finite simple group.

Non-asymptotic results on word maps are often very challenging and require more tools. See, for instance, \cite{LBST}, \cite{LOST2},
\cite{GM}, \cite{GT} and \cite{GLOST}. In \cite{LBST} it is shown that the commutator map is surjective on {\it all} finite simple
groups, proving the Ore Conjecture. Using a somewhat similar strategy, it was subsequently proved in \cite{LOST2} that, if $w = x^p$
for a prime $p \ne 3, 5$, then $w(G)^2 = G$ for {\it all} finite simple groups $G$. In \cite{GM} a stronger result is established,
yielding the same conclusion for $w = x^k$, where $k$ is any prime power or a power of $6$.

The above result was further extended in \cite{GLOST}. Consider the power words $w_1 = x^k$ and $w_2 = x^{\ell}$, where $k = p^aq^b$ for
primes $p, q$ and $\ell$ is an odd integer.
By celebrated theorems of Burnside and of Feit and Thompson, finite groups satisfying the identity $w_1$
or the identity $w_2$ are solvable. In particular, $w_1, w_2$ are not identities of any finite simple group.
It was proved in \cite{GLOST} that, $w_1(G)^2 = G$ and $w_2(G)^3 = G$ for {\it all} finite simple groups $G$.

Our main theorems, stated below, may be regarded as far-reaching extensions of the above mentioned results.
In particular, we derive similar conclusions for {\it every} word $w$ which is not an identity of any finite simple group,
sometimes with somewhat larger exponents.
Our strategy is to pose and study a more general problem.

\begin{defn}
A collection of non-empty subsets $S(G)\subseteq G$, one for each finite group $G$, is \emph{characteristic} if for every homomorphism
$\varphi\colon G\to H$ we have $\varphi(S(G))\subseteq S(H)$.	
\end{defn}

Some comments are in order.

\begin{enumerate}

\item[1.] For any characteristic collection $S$ and for any finite group $G$, $S(G)$ must be a fully invariant -- in particular, a characteristic --
subset of $G$ (take $\varphi\colon G \to G$); furthermore, $1 \in S(G)$ (take $\varphi$ to be the trivial endomorphism of $G$).

\item[2.] Trivial examples of characteristic collections are $S(G) = \{ 1 \}$ for all $G$ (the minimal collection) and $S(G) = G$ for all $G$
(the maximal collection).

\item[3.] Given characteristic collections $S, T$, define their product $ST$ by $ST(G) := S(G)T(G)$, which is easily seen to be
a characteristic collection. This binary operation is associative with an identity element (the minimal collection). Thus the set of all characteristic collections is a monoid, which is partially ordered by inclusion (where $S \subseteq T$ if $S(G) \subseteq T(G)$ for all $G$).

\item[4.] Let $w \in F_d$ be a word, and consider the word maps it defines on finite groups $G$. Then the collection $S(G) := w(G)$ 
is obviously characteristic.

\item[5.] Let $P_d$ denote the free profinite group on $x_1, \ldots , x_d$ (namely the profinite completion of $F_d$).
For any finite group $G$ and $g = (g_1, \ldots , g_d) \in G^d$ there is a unique homomorphism $\psi_g: P_d \to G$
satisfying $\psi_g(x_i) = g_i$ ($i = 1, \ldots , d$). Any element $W \in P_d$ (which may be regarded as a profinite word)
gives rise to a function (a profinite word map) $W: G^d \to G$ satisfying $W(g) := \psi_g(W)$.
Setting $S(G) := W(G)$, the image of $W$, we obtain a characteristic collection $S$. It is easy to see that different elements
$W_1, W_2 \in P_d$ give rise to distinct characteristic collections $S_1, S_2$. Hence the monoid of characteristic
collections has cardinality $2^{\aleph_0}$.

\item[6.] Any word $w$ in one variable, i.e. $w = x^k$ for some integer $k$, defines a word map $w\colon G\to G$ on each finite group $G$, 
and the collection of kernels $S(G) := w^{-1}(1) = \{ g \in G : g^k = 1 \}$ is characteristic.
\end{enumerate}

\begin{defn}
A characteristic collection $S$ is \emph{ample} if $|S(G)|\ge 2$ for all finite simple groups $G$.
\end{defn}

For example, let $w$ be a word, and let $S$ be its associated characteristic collection (i.e. $S(G) = w(G)$).  Then the ampleness of
$w$ is equivalent to each of the following conditions.
\begin{enumerate}[\rm(a)]
\item $w$ is not an identity of any finite simple group.
\item $w$ is not an identity of any minimal finite simple group
(these are certain well known groups of type $\PSL_2$ or Suzuki type).
\item For all finite groups $G$, $w(G) = 1$ implies that $G$ is solvable.
\end{enumerate}

\begin{defn}
We define the \emph{characteristic covering number} $\ccn(G)$ of a finite group $G$ to be the smallest integer $n$ such that
if $S_1,\ldots,S_n$ are ample characteristic collections, then $S_1(G)\cdots S_n(G) = G$.  If no such $n$ exists, we say $\ccn(G) = \infty$.
\end{defn}

The main result of this paper is the following:

\begin{thmA}\label{main1}
If $G$ is a finite simple group, then $\ccn(G)\le 6$.
\end{thmA}

Applying this for characteristic collections associated with words, we immediately obtain the following.

\begin{thmB}\label{main-ww1}
If $w_1,\ldots, w_6$ are words in disjoint letters, and none of the $w_i$ is an identity on any finite simple group, then the juxtaposition $w_1\cdots w_6$ is surjective on every finite simple group.
\end{thmB}

\begin{corC}\label{main-ww2}
If $w$ is a word which is not an identity on any finite simple group, then $w(G)^6 = G$ for every finite simple group $G$.
\end{corC}

Clearly, Theorem B and Corollary C also follow when $w, w_i$ are replaced by profinite words $W$, $W_i$.

Note that the condition that none of the $w_i$ (or $w$) is an identity on any finite simple group in Theorem B and
Corollary C is necessary. As shown in \cite{KaN}, \cite{GT}, for any integer $N$, there exist a word $w$ and a finite
simple group $G$ such that $w(G) \neq 1$ and $w(G)^N \neq G$.

As for lower bounds, it follows from \cite[8.9]{GLOST} that $\ccn(G) \ge 3$ for some finite simple groups $G$;
indeed, there are odd integers $\ell$ and finite simple groups $G$ such that $w(G)^2 \ne G$ for $w = x^{\ell}$.
This can be improved as follows.  In the special case that $w=x^2$, the collection $S(G) = w^{-1}(1)$ is ample by Feit-Thompson.  The minimal $n$
such that $S(G)^n = G$ is the \emph{involution width} of $G$, denoted $\iw(G)$.  This invariant has been investigated by several mathematicians (see \cite{M}
and the references therein).  In particular, Kn\"uppel and Nielsen showed \cite[16]{KN} that $\iw(\SL_n(q))\ge 4$ if $n\ge 5$ and $q\ge 7$.
When, in addition, $\gcd(n,q-1)=1$, we have $\SL_n(q) = \PSL_n(q)$, and this gives a lower bound of $4$ for the characteristic covering number of
infinitely many finite simple groups.

It would be interesting to know whether our upper bound of $6$ for $\ccn$ can be improved.  Malcolm showed \cite{M} that $\iw(G) \le 4$ for all finite simple groups,
and it seems quite possible that this is the optimal bound for $\ccn$ as well.  For most finite simple groups we can prove an upper bound of $4$ or less.  Indeed we have

\begin{thmD}\label{main2}
Let $G$ be a finite simple group. Then $\ccn(G)\le 4$ unless $G$ is a group of Lie type $X_r(q)$ where $q \le f(r)$
for a suitable function $f$.
\end{thmD}

Consequently, excluding these possible exceptions, we have $w_1(G) \cdots w_4(G) = G$ for all words $w_1, \ldots , w_4$
which are not an identity of any finite simple group.

Treating the remaining groups seems very challenging.

As $\ccn(G)$ reflects information on all simple subgroups of $G$, it cannot easily be determined from the character table of $G$, as $\iw(G)$ and similar invariants such as the
covering number and extended covering number of $G$ can be.  We were able to determine it in a few cases, of which the following results are representative.

\begin{prop}
\label{l2-prime}
If $p\ge 5$ is a prime, then
$$\ccn(\PSL_2(p)) =
\begin{cases}
2 &\text{if $p\equiv 1\pmod 4$,}\\
3 &\text{if $p\equiv 3\pmod 4$.}
\end{cases}$$
\end{prop}

\begin{prop}
\label{alt}
For all sufficiently large $n$, $\ccn(\AAA_n) = 3$.
\end{prop}

\begin{prop}
\label{sln-big}
If $n\ge 5$ and $q-1$ is sufficiently large and relatively prime to $n$, then $\ccn(\PSL_n(q)) = 4$.
\end{prop}

It would be interesting to find an example where $\ccn(G) > \iw(G)$; at present, we do not know of any.

\section{Preliminaries}

\begin{prop}
Let $G$ be a finite group.
If $\ccn(G) < \infty$, then $G$ is perfect.  If $G$ is simple, then $\ccn(G) < \infty$.
\end{prop}

\begin{proof}
For $G$ finite, $\ccn(G) < \infty$ if and only if $S(G)$ generates $G$ for every ample characteristic collection $S$.

If $G$ is not perfect, then by pulling back a subgroup of prime index from its abelianization, we obtain a normal subgroup $G_p\triangleleft G$ of index $p$ for some prime $p$.  Then $S(H) = \{h^p\mid h\in H\}$ defines a characteristic collection of subgroups, and it is ample since no finite simple group has exponent $p$.  However, $S(G)^n$ is contained in $G_p$ for all $n \in \N$, so $\ccn(G) = \infty$.

If $G$ is simple, let $S$ be an ample characteristic collection.  Then $S(G)$ generates a non-trivial normal subgroup of $G$.  As $S$ is ample, this must be $G$ itself.
\end{proof}

We recall that the extended covering number $\ecn(G)$ of a finite simple group $G$ is the smallest integer $n$ such that $C_1\cdots C_n = G$
if $C_1,\ldots,C_n$ is any sequence of non-trivial conjugacy classes of $G$.

\begin{lem}
\label{ecn}
If  $G$ is a finite simple group, then $\ccn(G)\le \ecn(G)-1$.
\end{lem}

\begin{proof}
If $\ccn(G) \ge n$, then there exist ample characteristic collections $S_i$ ($i = 1, \ldots , n-1$) such that
$S_1(G)\cdots S_{n-1}(G) \neq G$.
As $S_1,\ldots,S_{n-1}$ are ample,
$$S_1(G)\cdots S_{n-1}(G)\supseteq  C_1\cdots C_{n-1}$$
for some sequence $C_1,\ldots,C_{n-1}$ of non-trivial conjugacy classes of $G$.
Moreover
$$1\in S_1(G)\cdots S_{n-1}(G).$$
Therefore, there exists a non-trivial conjugacy class $D$ such that
$D$ is disjoint from $C_1\cdots C_{n-1}$.
This implies  $1\not\in C_1\cdots C_{n-1} D^{-1}$, so $\ecn(G) \ge n+1$.
\end{proof}

\def\ucn{\mathsf{ucn}}

\begin{lem}
\label{SL2}
Let $q\ge 4$ be a prime power and $C_1,C_2,C_3$ not necessarily distinct non-central conjugacy classes in $\SL_2(q)$.

\begin{enumerate}
\item[\rm(i)]The product $C_1 C_2 C_3$ contains all non-central elements of $\SL_2(q)$.
\item[\rm(ii)]If $q\neq 5$, then $C_1 C_2$ contains elements of  order  $q-1$ and  order  $q+1$.
\item[\rm(iii)]If $q=5$, then $C_1 C_2$ contains an element of order $4$ and one of order $3$ or order $6$.
\end{enumerate}
\end{lem}

\begin{proof}
Let $G$ be a finite group.
Recall that, by the extended Frobenius formula, the number of solutions to the equation $g = x_1x_2\cdots x_k$ for a fixed $g \in G$
and $x_i\in C_i = c_i^G$ ($1 \leq i \leq k$) is given by
\begin{equation}\label{frob}
\frac{\prod_{i=1}^k |C_i|}{|G|} \cdot \sum_{\chi \in \Irr(G)} \frac{\chi(c_1) \cdots \chi(c_k) \chi(g^{-1})}{\chi(1)^{k-1}}.
\end{equation}
This implies that $g \in C_1 \cdots C_k$ provided the above expression is non-zero.

The result now follows from the well known character table of $G: = \SL_2(q)$.
\end{proof}

\begin{lem}
\label{L2}
Let $q\ge 4$ be a prime power, $d$ the g.c.d. of $q-1$ and $2$, and $C_1,C_2,C_3$ not necessarily distinct non-central conjugacy classes in $G := \PSL_2(q)$.
Let $\CL_q$ denote the set of elements of order $q$.

\begin{enumerate}
\item[\rm(i)]The product $C_1 C_2 C_3$ equals $G$.
\item[\rm(ii)]The product $C_1 C_2$ contains elements of order $(q-1)/d$ and $(q+1)/d$
\item[\rm(iii)]If $q\equiv 1\pmod 4$ is prime, then $C_1 C_2$ contains an element of $\CL_q$.
\item[\rm(iv)]If $q\equiv 1\pmod 4$ is prime, then $C_1\CL_q$ contains $G\setminus \{1\}$.
\item[\rm(v)]If $q\equiv 3\pmod 4$ is prime and $C_1,C_2$ are classes of involutions, then $C_1C_2$ is disjoint from $\CL_p$.

\end{enumerate}

\end{lem}

\begin{proof}
For $q$ even,  $\PSL_2(q) = \SL_2(q)$, but in either case (i) and (ii) follow immediately from the corresponding statements in Lemma~\ref{SL2}.
For $q$ an odd prime, $\CL_q$ is the union of the two conjugacy classes.
Conclusions (iii)--(v) follow by computing (\ref{frob}) in the relevant cases using the character table of $G:=\PSL_2(q)$.
\end{proof}

From this lemma, we easily deduce Proposition~\ref{l2-prime}.
\begin{proof}[Proof of Proposition~\ref{l2-prime}]
The two conjugacy classes in $\CL_q$ are conjugate to one another under $\PGL_2(q)$, so if $S_1(G)S_2(G)$ meets $\CL_q$ at all, it contains it.  If $S_1(G)$ and $S_2(G)$
each contain a non-trivial conjugacy class of $G$, then by Gow's theorem \cite{Gow}, $S_1(G)S_2(G)$ contains all semisimple elements.  If $q\equiv 1\pmod 4$, then by Lemma~\ref{L2} (iii),
$S_1(G)S_2(G)=G$, so $\ccn(G)\le 2$.  On the other hand, $\ccn(G)\ge \iw(G)\ge 2$, so $\ccn(G)=2$.  If $q\equiv 3\pmod 4$, then Lemma~\ref{L2} (i) implies $\ccn(G)\le 3$,
and Lemma~\ref{L2} (v) implies $\iw(G)\ge 3$, so $\ccn(G)=3$.
\end{proof}

\begin{lem}
\label{PGL2}
If $q\ge 4$ and $C_1$ and $C_2$ are non-trivial $\PGL_2(q)$ orbits in $\PSL_2(q)$, then $C_1 C_2$ covers all split regular semisimple conjugacy classes in $\PSL_2(q)$.
\end{lem}

\begin{proof}
Again, this follows from easy character table computations.
\end{proof}

\begin{lem}
\label{L3}
\begin{enumerate}[\rm(i)]
\item If $C_1,C_2,C_3$ are non-trivial conjugacy classes in $\PSL_3(q)$, then $C_1C_2C_3$ contains all non-trivial elements of $\PSL_3(q)$.
\item If $2 < q \equiv -1(\bmod\ 3)$ and $C_1,C_2,C_3$ are non-trivial conjugacy classes in $\PSU_3(q)$, then $C_1C_2C_3$ contains all non-trivial elements of $\PSU_3(q)$.
\item If $3 \nmid (q+1)$ and $C_1,C_2,C_3$ are non-trivial conjugacy classes in $\SU_3(q)$, then $C_1C_2C_3$ contains all elements of order $q^2-q+1$ in $\SU_3(q)$. Moreover, $\ecn(\SU_3(q))=5$.
\end{enumerate}
\end{lem}

\begin{proof}
Parts (i) and (ii) follow immediately from the fact \cite[Corollary 1.9]{O} that $\ecn(\PSL_3(q)) = 4$, and that
$\ecn(\PSU_3(q))=4$ for those $q$.

(iii) In this case, $\GU_3(q) = \SU_3(q) \times \ZB(\GU_3(q))$, and so each $C_i=g_i^{\SU_3(q)}$ is a non-central $\GU_3(q)$-conjugacy class; furthermore, any element $g$ of order $q^2-q+1$ belongs to a class $C_8^{(k)}$ in \cite[Table 1]{O}. By assumption,
$\det(g_i)=1$ for $1 \leq i \leq 3$, and thus condition (2) of \cite[Theorem 1.3]{O} holds. Choosing $m=4$, we see that condition
(3) of \cite[Theorem 1.3]{O} also holds; see the first remark on \cite[p. 221]{O}. Now \cite[Theorem 1.3]{O} shows that
$1 \in C_1 C_2 C_3 \cdot (g^{-1})^{\SU_3(q)}$, and so $g \in C_1C_2C_3$. The last statement is contained in
\cite[Corollary 1.9]{O}.
\end{proof}

\begin{lem}
\label{Gow}
Let $G$ be a finite simple group of Lie type, and let $s,t \in G$ be two semisimple elements such that
\begin{enumerate}[\rm(i)]
\item For any pair of ample characteristic collections $S_1,S_2$, the conjugacy class of $s$ belongs to $S_1(G)S_2(G)$,
\item The product of the conjugacy class of $s$ and the conjugacy class of $t$ contains $G\smallsetminus\{1\}$,
\item The element $s$ is regular.
\end{enumerate}
Then $\ccn(G)\le 6$.
\end{lem}

\begin{proof}
By the theorem of Gow \cite{Gow}, every semisimple element in $G$, in particular $t$, can be written as a product of two conjugates of $s$.
Now, if $S_3,\ldots,S_6$ are ample characteristic collections, then
$s^G \subseteq S_3S_4$ and $s^G \subseteq S_5S_6$, whence $t^G$ lies in $S_3(G)S_4(G)S_5(G)S_6(G)$.
Therefore
$$G\smallsetminus\{1\}\subset s^G\cdot t^G \subseteq S_1(G)\cdots S_6(G),$$
and it follows that $\ccn(G)\le 6$.
\end{proof}

\begin{lem}
\label{scubed}
Let $G$ be a finite group,
and let $\dl(G)$ the lowest degree of a non-trivial character of $G$.
Let $q \geq 4$ and $\varphi\colon \SL_2(q)\to G$ be a non-trivial homomorphism. Suppose that
\begin{equation}
\label{prettyrss}
|\CB_G(\varphi(s))|^{3/2} \leq \dl(G)
\end{equation}
for any $s \in \SL_2(q)$ of order $q-1$, or for any $s \in \SL_2(q)$ of order $q+1$ if $q \neq 5$, or
for any $s \in \SL_2(q)$ of order $3$ or $6$ if $q=5$.
Then $\ccn(G) \le 6$.
\end{lem}

\begin{proof}
By Lemma~\ref{L2}, if $S_1,S_2$ are ample characteristic collections, then $g:=\varphi(s)\in S_1(G) S_2(G)$, where
$s \in \SL_2(q)$ can be chosen to have order $q-1$, or $q+1$ if $q \neq 5$, or either $3$ or $6$ when $q=5$.
Let $C:=g^G$. By the Frobenius formula, $C^3 = G$ if
$$\sum_{1_G \neq \chi \in \Irr(G)}\frac{|\chi(g)|^3}{\chi(1)} < 1.$$
For any $\chi$ in the summation, $|\chi(g)| \leq |\CB_G(g)|^{1/2}$ by the centralizer bound for character values, and
$\chi(1) \geq \dl(G)$. Hence, by the second orthogonality relation,
$$\sum_{1_G \neq \chi \in \Irr(G)}\frac{|\chi(g)|^3}{\chi(1)} < \frac{|\CB_G(g)|^{1/2}}{\dl(G)}\cdot \sum_{\chi \in \Irr(G)}|\chi(g)|^2
=  \frac{|\CB_G(g)|^{3/2}}{\dl(G)},$$
and we can deduce the needed inequality from \eqref{prettyrss}.
\end{proof}

We conclude this section with two results which allow us to prove that $\ccn(G)\le 4$ when $G$ is of Lie type and is sufficiently large compared to its rank.
\begin{prop}
\label{partition}
Let $\uG$ denote a connected, simply connected simple algebraic group of rank $r$ over $\F_q$ and $G$ the finite simple group obtained by taking the quotient of $\uG(\F_q)$ by its center.  Let $r = r_1+ \cdots + r_k$ be a partition.  If there exists a homomorphism $\phi\colon \SL_2(q^{r_1})\times \cdots \times \SL_2(q^{r_k})\to G$
with central kernel and if $q$ is sufficiently large in terms of $r$, then $\ccn(G)\le 4$.
\end{prop}

\begin{proof}
The number of $\F_q$-points on a torus of rank $s$ is at most $(q+1)^s$.  The regular elements in a maximal torus $\uT$ of $\uG$ lie in a union of
proper subtori of $\uT$ indexed by positive roots, so the number of elements in $\uT(\F_q)$ which are not regular semisimple is less than $2r^2(q+1)^{r-1}$.

We write $H := \prod_{i=1}^k \SL_2({q^{r_i}}) / Z$, where $Z:=\ker \phi$ is some subgroup of the product of the centers of the $\SL_2(\F_{q^{r_i}})$ and is therefore of
order at most $2^r$.  The inclusion of $H$ in $G$ gives a homomorphism
$$\prod_{i=1}^k \F_{q^{r_i}}^\times \to G.$$
The image under this homomorphism of elements which are non-trivial in each coordinate is at least
$$2^{-r} (1-2q^{-1})^r q^r.$$
However, the image lies in a maximal torus $T$ of $G$, i.e., the image of the $\F_q$-points of a maximal torus $\uT$ of $\uG$ in $G$.  Lemma~\ref{PGL2}
therefore gives a lower bound of the form $c_rq^r$ for the cardinality of the set of regular semisimple elements of $T$ which lie in $S_1(H) S_2(H)$ where $S_1$ and $S_2$ are ample characteristic collections.
The conjugacy class of each such element meets $T$ in at most $|W|$ elements, where $W$ denotes the Weyl group of $\uG$ with respect to $\uT$,
whose order is bounded in terms of $r$.  Each conjugacy class of a regular semisimple element in $\uG(\F_q)$ has at least $(q+1)^{-r} |\uG(\F_q)|$ elements,
so overall, we get a positive lower bound, depending only on $r$, for the proportion of elements of $G$ which lie in $S_1(G) S_2(G)$.
Since $S_1(G)S_2(G)$ is a normal subset of $G$, \cite[Theorem~A]{LST2} implies that $S_1(G)S_2(G)S_3(G)S_4(G)$ contains all non-trivial elements of $G$.
\end{proof}

Note that the same argument works also for Suzuki and Ree groups.

\begin{lem}
\label{Ellers-Gordeev}
Let $\uG$ be a connected, simply connected simple algebraic group of rank $r$ over $\F_q$ and $G$ the quotient of $\uG(\F_q)$ by its center.
Let $\phi\colon \SL_2^k\to \uG$ be a homomorphism of algebraic groups over $\F_q$ whose kernel lies in the center of $\SL_2^k$.
Let $\G_m^k$ denote a split maximal torus of $\SL_2^k$, and suppose $\phi(\G_m^k)$ is a maximal split torus of $\uG$ which contains regular $\F_q$-points.
Then if $q$ is sufficiently large in terms of $r$, we have $\ccn(G)\le 4$.
\end{lem}

Note that if $\uG$ is split, then $\phi(\G_m^k)$ is a maximal torus, so it contains regular elements over $\F_q$ if $q$ is sufficiently large.
If $\uG$ is not split, a maximal split torus is not a maximal torus, so $k<r$.

\begin{proof}
For $q$ sufficiently large, there exists $(a_1,\ldots,a_k)\in (\G_m(\F_q)\smallsetminus \{1\})^k$ such that $\phi(a_1,\ldots,a_k)$ is regular.
By Lemma~\ref{PGL2}, $(a_1,\ldots,a_k)\in S_1(\SL_2(q)^k)S_2(\SL_2(q)^k)$ if $S_1$ and $S_2$ are ample characteristic collections.
By a theorem of Ellers and Gordeev \cite[Theorem 1]{EG} the product of $G$-conjugacy classes of any two regular semisimple elements lying in
maximal split tori of $G$ covers $G\smallsetminus\{1\}$.
\end{proof}

\bigskip

\section{Classical groups}

For every positive integer $n$, we denote by $\Phi_n(x)$ the $n$th cyclotomic polynomial.  We recall \cite{Zs}
that for $n>2$ and $(n,q)\neq (6,2)$, $\Phi_n(q)$
is always divisible by a Zsigmondy prime $\ell$, meaning that $\ell$ does not divide $q$, which implies and the order of $q$ (mod $\ell$) is exactly $n$.  In particular,
$\ell\equiv 1\pmod n$.

\begin{thm}
\label{Classical}
If $G$ is a finite simple classical group, then $\ccn(G)\le 6$.
\end{thm}

\begin{proof}
We consider all the six types $A_r$, $\tw2 A_r$, $B_r$, $C_r$, $D_r$, and $\tw2 D_r$. Let $S_1, \ldots,S_6$
be ample characteristic collections and set $d:=\gcd(2,q-1)$.
\medskip

\noindent{\bf Case} $G=\PSL_n(q)$ with $n \geq 2$ and $(n,q) \neq (2,2)$, $(2,3)$.
Note that $\ecn(\PSL_2(q)) =4$ for $q \geq 4$ \cite[p. 2]{AH}, and
$\ecn(\PSL_3(q)) =4$ \cite[Corollary 1.9]{O}, so we may assume $n \geq 4$.
For $(n,q) = (6,2)$, $(7,2)$, we have $\SL_2(8) < G$, and any element of order $7$ lies in $S_1(G) S_2(G)$ for
$S_i$ ample characteristic collections.  Using the character tables for these two groups in the GAP library, we check that the
square of this conjugacy class covers the non-trivial elements of the group, so $\ccn(G)\le 4$.

We may therefore assume that $(n,q) \neq (6,2)$, $(7,2)$.
Write $q=p^f$ for a prime $p$, and $m := \lfloor n/2 \rfloor$ so that $m \geq 2$ and $n\in \{2m,2m+1\}$.
By \cite{Zs} and the assumption on $(n,q)$, we can find a Zsigmondy prime $\ell$ for $(2mf,p)$.
By Lemma \ref{L2}, $S_1(\PSL_2(q^m))S_2(\PSL_2(q^m))$ contains an element of order $(q^m+1)/d$.
Since $\SL_2(q^m) \leq \SL_{2m}(q) \leq \SL_n(q)$, it follows that $S_1(G)S_2(G)$ contains an element
$s$ of order divisible by $(q^m+1)/d$ which is divisible by $\ell$, and so $s$ is regular semisimple. As discussed in \cite[\S2.2.1]{GT},
one can find a semisimple element $t \in G$ such that $s^G\cdot t^G \supseteq G \smallsetminus \{1\}$. Hence $\ccn(G) \leq 6$ by
Lemma \ref{Gow}.
\vskip 0.2cm\noindent{\bf Case} $G=\PSU_n(q)$ with $n \geq 3$ and $(n,q) \neq (3,2)$.
Note that $\ecn(\PSU_3(q)) \leq 5$ \cite[Corollary 1.9]{O}, so we may assume $n \geq 4$.
If $G=\PSU_6(2)$, then $\SL_2(8)<G$, and we easily check that the square of the class in $G$ of any element of order $7$
covers the nontrivial elements of $G$, so $\ccn(G)\le 4$.

We therefore assume that $(n,q) \neq (6,2)$.
Again write $m := \lfloor n/2 \rfloor$, so that $m \geq 2$ and $n\in \{2m,2m+1\}$.
By \cite{Zs} and the assumption on $(n,q)$, we can find a Zsigmondy prime $\ell$ for $(2mf,p)$ when $2|m$,
and for $(mf,p)$ when $2 \nmid m$; note that $\ell$ divides $q^{2m}-1$, but not
$\prod^{2m-1}_{i=1}((-q)^i-1)$. By Lemma \ref{L2}, $S_1(\PSL_2(q^m))S_2(\PSL_2(q^m))$ contains an element of order
$(q^m+(-1)^m)/d$. Note that
$$\SL_2(q^m) \leq \SL_{m}(q^2) < \SU_{2m}(q) \leq \SU_n(q)$$
when $2|m$, and
$$\SL_2(q^m) \cong \SU_2(q^m) \leq \SU_{2m}(q) \leq \SU_n(q)$$
when $2 \nmid m$. It follows that $S_1(G)S_2(G)$ contains an element
$s$ of order divisible by $(q^m+(-1)^m)/d$ which is divisible by $\ell$, and so $s$ is regular semisimple (as one can see by checking
the eigenvalues of $s$). As discussed in \cite[\S2.2.2]{GT},
one can find a semisimple element $t \in G$ such that $s^G\cdot t^G \supseteq G \smallsetminus \{1\}$. Hence $\ccn(G) \leq 6$ by
Lemma \ref{Gow}.
\vskip 0.2cm\noindent{\bf Case} $G=\PSp_{2n}(q)$ with $n \geq 2$ and $(n,q) \neq (2,2)$.
(Note that $\Sp_4(2)' \cong \PSL_2(9)$). By Lemma \ref{L2}, $S_1(\PSL_2(q^n))S_2(\PSL_2(q^n))$ contains an element $\bar{s}$
of order $(q^n-1)/d$ and $S_3(\PSL_2(q^n))S_4(\PSL_2(q^n))$ an element $\bar{t}$ of order $(q^n+1)/d$.
Since
$$\SL_2(q^n) \cong \Sp_2(q^n) < \Sp_{2n}(q),$$
$S_1(G)S_2(G)$ contains regular semisimple elements
$s$ of order divisible by $(q^n-1)/d$, and $S_3(G)S_4(G)$ contains a regular semisimple element of order divisible by $(q^n+1)/2$. As discussed in \cite[\S2.2.3]{GT}, $s^G\cdot t^G \supseteq G \smallsetminus \{1\}$. Hence
$S_1(G)S_2(G)S_3(G)S_4(G) \supseteq G \smallsetminus \{1\}$, and $\ccn(G) \leq 4$.
\vskip 0.2cm\noindent{\bf Case} $G=P\Omega^\e_n(q)$ with $n \geq 7$ and $\e = \pm$.
First assume that $n=7$ and $2 \nmid q$. Note that
$$\Omega_7(q) > \SO^+_6(q) > \SL_3(q),~\Omega_7(q) > \SO^-_6(q) > \SU_3(q).$$
Hence, by Lemma \ref{L3}, $S_1(G)S_2(G)S_3(G)$ contains a regular semisimple element $s$ of order $q^2+q+1$ and
$S_4(G)S_5(G)S_6(G)$ contains a regular semisimple element $t$ of order divisible by $(q^2-q+1)/\gcd(3,q+1)$. As discussed in \cite[\S2.2.3]{GT}, $s^G\cdot t^G \supseteq G \smallsetminus \{1\}$. Hence $\ccn(G) \leq 6$.

We may now assume $n \geq 8$, so that $m:=\lfloor n/4 \rfloor \geq 2$. Note that
$$\PSL_2(q^{2m}) \cong \Omega^-_4(q^m) \leq \Omega^-_{4m}(q),$$
and $S_1(\PSL_2(q^{2m}))S_2(\PSL_2(q^{2m}))$ contains an element $\tilde s$ of order $(q^{2m}+1)/d$. Such an element $s$
is regular semisimple in each of the terms of
$$\Omega^-_{4m}(q) < \Omega_{4m+1}(q) < \Omega^\pm_{4m+2}(q).$$
Hence, if $G \in \{P\Omega^-_{4m}(q),\Omega_{4m+1}(q),P\Omega^\pm_{4m+2}(q)\}$, then
$S_1(G)S_2(G)$ contains a regular semisimple element $s$ of order $(q^{2m}+1)/d$. As discussed in \cite[\S\S2.2.3, 2.2.4]{GT},
one can find a semisimple element $t \in G$ such that $s^G\cdot t^G \supseteq G \smallsetminus \{1\}$. Hence $\ccn(G) \leq 6$ by
Lemma \ref{Gow}. In particular, we are done with type $\tw2 D_r$.

Consider the case of $G = P\Omega^+_{4m}(q)$, and note that
$$\PSL_2(q^2) \times \PSL_2(q^{2m-2}) \leq \Omega^-_4(q) \times \Omega^-_{4m-4}(q) < \Omega^+_{4m}(q).$$
Hence, if $m \geq 3$, then applying Lemma \ref{L2} we see that $S_1(G)S_2(G)$ contains an element
$s=(x,y)$, with $x \in S_1(\PSL_2(q^2))$ of order $(q^2+1)/d$ and $y \in S_2(\PSL_2(q^{2m-2}))$ of order
$(q^{2m-2}+1)/d$. This element $s$ is regular semisimple of type $T^{-,-}_{2,2m-2}$, and, as shown in
\cite[Proposition 7.1.1]{LST} and \cite[Theorem 7.6]{GM}, there exists a semisimple element
$t \in G$ such that $s^G\cdot t^G \supseteq G \smallsetminus \{1\}$. Hence $\ccn(G) \leq 6$ by
Lemma \ref{Gow}. Suppose now that $m=2$, but $q \geq 3$. Then
$$\SL_3(q) < \SO^+_6(q) \hookrightarrow \Omega^+_8(q),~\SU_3(q) < \SO^-_6(q) \hookrightarrow \Omega^+_8(q).$$
By Lemma \ref{L3}, $S_1(G)S_2(G)S_3(G)$ contains an element $s$ of order divisible by $(q^2+q+1)/\gcd(3,q-1)$, and
such an element is regular semisimple of type $T^{+,+}_{3,1}$ in $G$, see \cite[\S2.1]{GT}. Likewise,
$S_4(G)S_5(G)S_6(G)$ contains an element $t$ of order divisible by $(q^2-q+1)/\gcd(3,q+1)$, and
such an element is regular semisimple of type $T^{-,-}_{3,1}$ in $G$ (since $q \geq 3$). By \cite[Lemma 2.4]{GT},
$s^G\cdot t^G = G \smallsetminus \{1\}$, whence $\ccn(G) \leq 6$. In the case $G=\Omega^+_8(2)$, one can
check directly that $\ccn(G)\leq 4$, using an element of order $7$ in $\SL_2(8) < G$.

Consider the case of $G = \Omega_{4m+3}(q)$ with $m \geq 2$, and note that
$$\PSL_2(q^{2m}) < \Omega^-_{4m}(q) \times \Omega_3(q) < \Omega^+_{4m+3}(q).$$
Again using Lemma \ref{L2} we see that $S_1(G)S_2(G)$ contains an element
$s  \in S_1(\PSL_2(q^{2m}))$ of order $(q^{2m}+1)/d$. Note that
$$\CB_{\SO_{4m+3}(q)}(s) = T \times \SO_3(q),$$
where $T < \SO^-_{4m}(q)$ has order $q^{2m}+1$. Since $\SO_3(q) \cap G = \Omega_3(q)$ has index $2$ in
$\SO_3(q)$, it follows that
$$|\CB_G(s)| =  (q^{2m}+1)(q^3-q)/2.$$
Let $B:=q^{8m-4}$. By \cite[Corollary 5.8]{LBST}, $\Omega$ has exactly $q+4$ nontrivial irreducible characters of
degree $\leq B$, which are the characters $D^\circ_\alpha$ with $\alpha \in \Irr(\Sp_2(q))$,
listed in \cite[Proposition 5.7]{LBST}. The proof of Lemma \ref{scubed} shows that
\begin{equation}\label{so1}
  \sum_{\chi \in \Irr(G),~\chi(1) > B}\frac{|\chi(s)|^3}{\chi(1)} < \frac{|\CB_G(s)|^{3/2}}{B} \leq \frac{\bigl(q^{2m}+1)(q^3-q)/2\bigr)^{3/2}}{q^{8m-4}}
  < \frac{2^{-3/2}}{q^{5m-17/2}} < 0.07.
\end{equation}
The degrees of $D^\circ_\alpha$ are listed in \cite[Table I]{LBST}, showing that two of them have $\ell$-defect $0$ for $\ell$ a
Zsigmondy prime for $(n-3,q)$ and so vanish at $s$. Next we estimate $|\chi(s)|$ for the remaining $q+2$ characters.
Note that, in the action of $s$ on the natural $G$-module $\F_q^n$, $s$ has a unique eigenvalue $\lambda$ that belongs to
$\F_{q^2}$, and this eigenvalue is $\lambda=1$ and has multiplicity $3$. Consider the action of $x\otimes s$ on
$V:=\F_q^2 \otimes \F_q^n$ for any $x \in \Sp_2(q)$ and $\F_q^2$ being the natural module for $\Sp_2(q)$. Then the
fixed point subspace of $x \otimes s$ on $V$ has dimension $6$ if $x=1$, $3$ if $1 \neq x$ is unipotent, and
$0$ otherwise. It follows from the formula \cite[Lemma 5.5]{LBST} for $D_\alpha$ that
\begin{equation}\label{so2}
  |D_\alpha(s)| \leq \frac{\alpha(1)}{q(q^2-1)}\bigl(q^3+ q^{3/2}(q^2-1)+q(q^2-1)-q^2\bigr) = \alpha(1)(q^{1/2}+2q/(q+1)).
\end{equation}
Note that $D^\circ_\alpha=D_\alpha$, unless $\alpha(1)=(q+1)/2$ in which case $D^\circ_\alpha=D_\alpha-1_G$. Now, if
$q=3$, then $\alpha(1) \leq 3$, and so \eqref{so2} shows that $|D^\circ_\alpha(s)| < 10.5$. Since $D^\circ_\alpha(1) \geq 7260$,
it follows that
$$\sum_{\chi \in \Irr(G),~1 < \chi(1) \leq B}\frac{|\chi(s)|^3}{\chi(1)}  \leq \frac{5 \cdot (10.5)^3}{7260} < 0.8.$$
If $q \geq 5$, then since $\alpha(1) \leq q+1$, \eqref{so2} shows that $|D^\circ_\alpha(s)| \leq 2q+q^{1/2}(q+1)$. As
$D^\circ_\alpha(1) > q^8$, it follows that
$$\sum_{\chi \in \Irr(G),~1 < \chi(1) \leq B}\frac{|\chi(s)|^3}{\chi(1)}  \leq \frac{(q+2) \cdot \bigl(2q+q^{1/2}(q+1)\bigr)^3}{q^8} < 0.2.$$
Together with \eqref{so1}, we have shown that
$$\sum_{1_G \neq \chi \in \Irr(G)}\frac{|\chi(s)|^3}{\chi(1)} < 0.87,$$
and so
$(s^G)^3 \supseteq G \smallsetminus \{1\}$ by the Frobenius formula. Consequently, $\ccn(G) \leq 6$.
\end{proof}

\begin{prop}
\label{big-classical}
If $G$ is a classical group and the order of $G$ is sufficiently large in terms of the rank of $G$, then $\ccn(G)\le 4$.
\end{prop}
\begin{proof}We do a case analysis.

\vskip 0.2cm\noindent{\bf Case} $G=\PSL_n(q)$.  Let $m=\lfloor n/2\rfloor$.
The  obvious homomorphism $\SL_2(q)^m\to \SL_n(q)$ maps $(x_1,\ldots,x_m)$ to a regular semisimple element in a split maximal torus of $\SL_n$ over $\F_q$ and hence
to a split regular semisimple element of $G$,
whenever the $x_i$ are regular semisimple elements with eigenvalues $\lambda_i^{\pm 1}$,
$\lambda_i\in \F_q^\times$, and $\lambda_i\neq \lambda_j^{\pm 1}$ for all $i,j$.
This case now follows from Lemma~\ref{Ellers-Gordeev}.

\vskip 0.2cm\noindent{\bf Case} $G=\PSU_n(q)$.  Let $m=\lfloor n/2\rfloor$.
We proceed exactly as before, using the obvious homomorphism from
$\SL_2(q)^m = \SU_2(q)^m$ to $\SU_n(q)$.  An $m$-tuple of regular
semisimple elements such that $\lambda_i\neq \lambda_j^{\pm 1}$ for all $i,j$ maps to an
element in an $m$-dimensional split torus of $G$, and this is a maximal split torus \cite[Table 2]{Tits}.

\vskip 0.2cm\noindent{\bf Case} $G=\PSp_{2n}(q)$.  We proceed as before,
using the obvious homomorphism from $\SL_2(q)^n \cong \Sp_2(q)^n$ to $\Sp_{2n}(q)$,
which maps any split maximal torus of $\SL_2^n$ to a split maximal torus of $\Sp_{2n}$.

\vskip 0.2cm\noindent{\bf Case} $G=P\Omega_{2n+1}(q)$.  If $n=2m$, we proceed as before, using composition of the obvious homomorphisms from $\SL_2(q)^n$ to
$$\Spin_5(q) * \underbrace{\Spin^+_4(q)* \ldots * \Spin^+_4(q)}_{m-1} \cong \Sp_4(q) * \underbrace{\SL_2(q) * \ldots * \SL_2(q)}_{n-2}$$
and $\Spin_5(q)  * \underbrace{\Spin^+_4(q)* \ldots * \Spin^+_4(q)}_{m-1} \to \Spin_{2n+1}(q)$.
If $n=2m+1$, we map $\SL_2(q)^n$ onto
$\Spin_3(q)* \underbrace{\Spin^+_4(q)* \ldots * \Spin^+_4(q)}_{m}$ and maps the latter
to $\Spin_{2n+1}(q)$.

\vskip 0.2cm\noindent{\bf Case} $G=P\Omega_{2n}^+(q)$.  If $n=2m$, we map $\SL_2(q)^n$ onto
$\underbrace{\Spin^+_4(q)* \ldots * \Spin^+_4(q)}_{m}$ and map the latter to $\Spin_{2n}^+(q)$.  If $n=2m+1$, we map $\SL_2(q)^{n-1}$ to
$$\SL_4(q)* \underbrace{\SL_2(q)* \ldots * \SL_2(q)}_{n-3} \cong \Spin_6^+(q)* \underbrace{\Spin_4^+(q) * \ldots * \Spin_4^+(q)}_{m-1},$$
which embeds in the usual way in $\Spin_{2n}^+(q)$.

\vskip 0.2cm\noindent{\bf Case} $G=P\Omega_{2n}^-(q)$.
If $n=2m+1$, we map $\SL_2(q)^{n-1} = \SU_2(q)^2\times \SL_2(q)^{n-3}$ onto
$$\SU_4(q)* \underbrace{\SL_2(q) * \ldots * \SL_2(q)}_{n-3} \cong \Spin_6^-(q) * \underbrace{\Spin_4^+(q)* \ldots * \Spin^+_4(q)}_{m-1},$$
which embeds in the usual way in $\Spin_{2n}^-(q)$.  Note that an ordered pair of regular semisimple elements of $\SU_2(q)$
with distinct eigenvalue pairs gives a regular semisimple element of $\SU_4(q)$, and it follows that the image of a maximal torus of $\SL_2^{n-1}$
in $\Spin_{2n}^-(q)$ meets the regular semisimple locus.  As $\Spin_{2n}^-$ is not split, it cannot have a split torus of rank $n$, so any split torus of rank $n-1$ must be a maximal split torus.  It follows that the image of a split maximal torus of $\SL_2^{n-1}$ is a maximal split torus of
$\Spin_{2n}^-$.

If $n=2m$, we
identify $\Spin_4^-(q)$ with $\SL_2(q^2)$ and map $\SL_2(q^2)\times \SL_2(q)^{n-2}$ into $\Spin_{2n}^-(q)$.
This maps $(y,x_1,\ldots,x_{n-2})$ to a regular semisimple element when $y$ and the $x_i$ are regular semisimple, the eigenvalues of all $x_i$ lie in $\F_q$,
the eigenvalues of $y$ do not lie in $\F_q$, and no two $x_i$ have an eigenvalue in common.  We can no longer use the Ellers-Gordeev method,
but Proposition~\ref{partition} applies, so we still obtain $\ccn(G)\le 4$.
\end{proof}

\bigskip

\section{Exceptional groups of Lie type}

The following lemma will be useful for us.

\begin{lem}
\label{Zsig}
Let $G$ be an exceptional group of Lie type and $s$ an element of $G$ whose order is divisible by a prime $\ell$ which is Zsigmondy for $(n,q)$. Then
\begin{enumerate}[\rm(i)]
\item If $G = F_4(q)$ and $n=8$, then $s$ is regular semisimple with centralizer order $q^4+1$.
\item If $G = \tw2 F_4(q)$ with $q \geq 8$, $n=4$, and $\ell > 5$, then $s$ is regular semisimple with centralizer order $q^2+1$.
\item If $G = E_6(q)$ or $G=\tw2 E_6(q)$, and $n=8$, then $s$ is regular semisimple with centralizer order $(q^2-1)(q^4+1)$.
\item If $G=E_7(q)$ and $n=7$, then $s$ is regular semisimple with centralizer order $q^7-1$.
\item If $G=E_8(q)$ and $n=7$, then $s$ has centralizer order dividing $(q^3-q)(q^7-1)$.
\end{enumerate}
\end{lem}

\begin{proof}
In all five cases, $\ell-1$ has a divisor $\geq 5$, so $\ell\ge 11$.  Let $t$ denote a power of $s$ of exact order $\ell$, so $t$ is semisimple.
By \cite[Lemma~2.2]{MT}, the connected center $T$ of the centralizer of $t$ has order divisible by $\ell$.  It suffices to prove the stated claims for $t$, since the centralizer of $t$ contains the centralizer of $s$.  In general, $\ell$ divides $\Phi_m(q)$ only if $m$ is $n$ times a power of $\ell$.  The order of every torus $T$ in $G$ can be written $\prod_j \Phi_{i_j}(q)$, where $\sum \phi(i_j)$ is the rank of the torus; in particular, $\phi(i_j)\le 8$, which implies that $i_j < n\ell$, so if $\ell$ divides $|T$, then $i_j = n$ for some $i_j$.

In the $F_4$ case, $\phi(8) = 4$ so $T$ must have order $\Phi_8(t)$, so $t$ must be regular semisimple.
The same argument applies to $\tw2 F_4(q)$ (which is $\tw2 F_4(r^2)$ with $r=q^{1/2}$).
In the $E_6$ and $^2E_6$ cases, we consult the connected centralizers in the $E_6$ table in \cite{FJ1} and conclude that $t$ must be regular semisimple and associated to $w_{19}$, implying the stated centralizer order.
The $E_7$-table in \cite{FJ1} has a $\Phi_7(q)$ factor in the connected centralizer of $t$ only for the regular semisimple case associated to $w_{39}$.
The $E_8$-table in \cite{FJ2} has a $\Phi_7(q)$ factor in three cases: the regular semisimple classes associated to $w_{37}$ and $w_{55}$, with centralizer orders $(q-1)(q^7-1)$
and $(q+1)(q^7-1)$ respectively and the $(A_1,w_{51})$ class, with centralizer order $(q^3-q)(q^7-1)$.
\end{proof}

\begin{thm}
\label{Exceptional}
If $G$ is a finite simple group of exceptional Lie type, then $\ccn(G)\le 6$.
\end{thm}

\begin{proof}
We consider each of the ten possibilities.
\vskip 0.2cm\noindent{\bf Case} $^2B_2(q)$.  As $G$ is a finite simple group, we have $q\ge 8$.  By \cite[p.~2]{AH}, $\ecn(G) = 4$, so by Lemma~\ref{ecn}, $\ccn(G)\le 3$.
\vskip 0.2cm\noindent
{\bf Case} $G_2(q)'$.  The character table for $q=2$ is in the GAP library, and we easily compute that the extended covering number of this group is $5$, so
$\ccn(G) \le 4$.  Otherwise, $q\ge 3$, so $G_2(q)' = G_2(q)$.
By \cite[Table 5.1]{LSS}, we have $\SL_3(q) < G_2(q)$.  Let $e:=\gcd(3,q-1) = \gcd(3,q^2+q+1)$.
By Lemma~\ref{L3}, for all $S_1,S_2,S_3$, every element $\bar{s}$ of order $\frac{q^2+q+1}e$ in $\PSL_3(q)$ belongs to $S_1(\PSL_3(q))S_2(\PSL_3(q))S_3(\PSL_3(q))$, and let $s$ be any lift of $\bar{s}$ to $\SL_3(q)$.
By \cite[Table 10]{GM}, every non-trivial element of $G$ is a product of two such elements $s$, so $\ccn(G_2)\le 6$.
\vskip 0.2cm\noindent
{\bf Case} $^2G_2(q)'$.  When $q=3$, this is $\PSL_2(8)$, so we have $\ccn(G)\le 3$.  When $q\ge 27$, $^2G_2(q)$ is already simple.  By \cite[Table 5.1]{LSS},
$\PSL_2(q)<G$, so by Lemma~\ref{L2}, any $S_1(G)S_2(G)$ contains an element $s$ of order $\frac{q+1}2$, and by \cite[Theorem~7.1]{GM}, every element of $G\smallsetminus\{1\}$
is a product of two conjugates of $s$.  Thus $\ccn(G)\le 4$.
\vskip 0.2cm\noindent
{\bf Case} $^3D_4(q)$.  We compute
\begin{align*}
\ecn({}^3D_4(2)) = 7,~
\ecn({}^3D_4(3)) = 6,
\end{align*}
using the character table in the GAP library for the former and Frank L\"ubeck's character table, computed using the generic character table for $^3D_4$ in CHEVIE \cite{GHLMP}, for the latter.
So we may assume $q\ge 4$.

By \cite[Table 5.1]{LSS}, we have $\SL_2(q)\times \SL_2(q^3)<{}^3D_4(q)$ if $q$ is even and the central product $\SL_2(q)\ast \SL_2(q^3)$ if $q$ is odd.
In either case, by Lemma~\ref{SL2}, given ample characteristic collections $S_1$ and $S_2$, this subgroup contains an element $s$ of $S_1(G) S_2(G)$ of order
$(q-1)(q^3+1)/d$ where $d:=\gcd(2,q-1)$.  This element is regular semisimple of type $s_{11}$ in the Deriziotis-Michler classification \cite[Table 2.1]{DM};
it follows that the order of its centralizer is $(q-1)(q^3+1) < q^4$.  From L\"ubeck's table of degrees for $^3D_4(q)$ \cite{Lu}, we know that except for the trivial character and a character $\chi_1$ of degree $q\Phi_{12}(q)$, all other irreducible characters of $\tw3 D_4(q)$ have degree greater than $q^8/2$ when $q\ge 4$.  By \cite[Table 2]{Sp}, the value of $\chi_1$ at a regular semisimple element with centralizer $(q-1)(q^3+1)$ is $1$. Certainly, $\sum_{\chi \in \Irr(G)}|\chi(s)|^2 = |\CB_G(s)| < q^4$, and
$|\chi(s)| < ((q-1)(q^3+1))^{1/2} < q^2$ for all $1_G \neq \chi \in \Irr(G)$.
We conclude that
\begin{align*}
\sum_{\chi \neq 1} \frac{|\chi(s)|^3}{\chi(1)} & <  \frac 1{q\Phi_{12}(q)} + \sum_{\chi(1)>\epsilon_1(1)} \frac{|\chi(s)|^3}{\chi(1)} \\
                                                                     & < \frac 2{q^5} +\frac{q^2}{q^8/2}\sum_\chi |\chi(s)|^2 \\
                                                                     & < \frac2{q^5} + \frac 2{q^2} < 1,
\end{align*}
so the Frobenius formula implies that $(s^G)^3=G$.
\vskip 0.2cm\noindent
{\bf Case} $F_4(q)$.  By \cite[Table 5.1]{LSS}, $\Sp_4(q^2)<F_4(q)$.  Therefore, $\SL_2(q^4) < F_4(q)$.
We are therefore guaranteed elements $s$ of order $q^4+1$ in $S_1(F_4(q))S_2(F_4(q))$.  By Lemma~\ref{Zsig}, such elements are regular semisimple.
By \cite[Table 11]{GM}, If $t$ is any element of order $\Phi_{12}(t)$, then the product of the conjugacy class of $s$ and the conjugacy class of $t$ covers $F_4(q)\smallsetminus\{1\}$.
Thus, Lemma~\ref{Gow} implies $\ccn(F_4(q))\le 6$.

\vskip 0.2cm\noindent
{\bf Case} $^2F_4(q)'$.  For $q=2$, this is the Tits group.  In this case,
we can use the character table in GAP to compute the extended covering number,
which is $5$, so $\ccn(G)\le 4$.
We may therefore assume $q\ge 8$, so $^2F_4(q)$ is already perfect.
By \cite[Table 5.1]{LSS}, $\SL_2(q^2) < \Sp_4(q)<{}^2F_4(q)$.  Thus, $^2F_4(q)$ contains regular semisimple elements $s$ of order $q^2+1$ by Lemma \ref{Zsig}.
The centralizer of $s$ in $^2F_4(q)$ has order $q^2+1$, and by \cite{LZ} we have $\dl(G) \ge 2^{-1/2}(1-q^{-1})q^{11/2}$.
%
%
Lemma~\ref{scubed} now implies that $\ccn(G)\le 6$ for $q\ge 8$.

\vskip 0.2cm\noindent
{\bf Cases} $E_6(q)$ and $^2E_6(q)$.  As $F_4(q)<G$, we have $\SL_2(q^4) < G$, and proceeding as in the $F_4(q)$ case, we are guaranteed elements $s$ of order $q^4+1$ in
$S_1(G) S_2(G)$.  These elements are regular semisimple by Lemma~\ref{Zsig}.  By \cite[Table 11]{GM}, there exists a semisimple element $t$ such that the product of the conjugacy classes of $s$ and $t$ cover $G\smallsetminus\{1\}$.  Thus, $\ccn(G)\le 6$ by Lemma~\ref{Gow}.

\vskip 0.2cm\noindent
{\bf Case} $E_7(q)$.  By \cite[Table 5.1]{LSS}, $E_7(q)$ contains $\PSL_2(q^7)$.  If $S_1$ and $S_2$ are ample characteristic collections, then $S_1(\PSL_2(q^7)) S_2(\PSL_2(q^7))$
contains an element $s$ of order $q^7-1$ or $\frac{q^7-1}2$.  By Lemma~\ref{Zsig}, this element is regular semisimple.
In \cite[Table 11]{GM} is it shown that there exists an order $x_2$ prime to $q$ such that the product of any conjugacy class of order $x_1 = q^7-1$ and any conjugacy class of order $x_2$ covers $G\smallsetminus\{1\}$.  The proof depends only on the divisibility of $x_1$ by a Zsigmondy prime and therefore goes through unchanged if $x_1= \frac{q^7-1}2$.  Lemma~\ref{Gow} now implies that $\ccn(E_7(q))\le 6$.

\vskip 0.2cm\noindent
{\bf Case} $E_8(q)$. By \cite[Table 5.1]{LSS}, $E_8(q)$ contains a $E_7^{\sc}(q)$, which, in turn, contains a perfect central extension $\tilde H$ of $H =\PSL_2(q^7)$.
We can therefore regard $\tilde H$ as a central quotient of $\SL_2(q^7)$.  If $S_1$ and $S_2$ are ample characteristic collections, then $S_1(\tilde H) S_2(\tilde H)$
contains an element $s$ of order $q^7-1$ or $\frac{q^7-1}2$.  This element is therefore semisimple, and by Lemma~\ref{Zsig}, the order of its centralizer divides $(q^3-q)(q^7-1)$.
By \cite{LZ}, $\dl(G) \ge q^{27}(q^2-1)\ge q^{28}$.  As $(q^{10})^{3/2}  = q^{15} < q^{28}$, Lemma~\ref{scubed} implies $\ccn(E_8(q))\le 6$.
\end{proof}

In most cases, we can prove the improved bound of $4$.

\begin{prop}
\label{big-exceptional}
If $G$ is a sufficiently large finite simple group of exceptional Lie type, then $\ccn(G)\le 4$.
\end{prop}

\begin{proof}
For each series of exceptional groups of Lie type with the exception of $^2B_2$ (which is already covered by Theorem~\ref{Exceptional}), $E_6$ and $^2E_6$,
\cite[Table~5.1]{LSS} gives a subgroup of the type for which
Proposition~\ref{partition} applies.

Suppose $G$ is of type $E_6$, i.e., it is the quotient of $\uG(\F_q)$ by its center, where $\uG$ is the split simply connected group group over $\F_q$ of type $E_6$.
As $A_1\times A_5$ can be obtained from the extended Dynkin diagram of $E_6$ by deleting a vertex, it follows that
there exists a homomorphism  $\SL_2\times \SL_6\to \uG$ of algebraic groups over $\F_q$, and a maximal torus of the former maps to a maximal torus of the later.
Thus, there exists a homomorphism
$\phi\colon \SL_2^4\to \uG$ which factors through $\SL_2\times \SL_6\to \uG$, which maps a split maximal torus into a split maximal torus, and whose image contains regular
elements if $q$ is sufficiently large.  The proposition now follows from Lemma~\ref{Ellers-Gordeev}.

Finally suppose $G$ is of type $^2E_6$.  Let $\uG$ be a simply connected non-split group of type $E_6$ over $\F_q$.  By \cite[Table 2]{Tits}, it is of rank $4$ over $\F_q$.
The rank of $\SL_2\times \SU_6$ over $\F_q$ is also $4$, so there is a homomorphism $\SL_2\times \SU_6\to \uG$ mapping a maximal split torus to a maximally split maximal torus.
Using $\SL_2^3 = \SU_2^3\to \SU_6$, we obtain a morphism of algebraic groups over $\F_q$ from $\SL_2^4$ to $\uG$ which sends any split maximal torus of $\SL_2^4$ to a maximal split torus of $\uG$ and whose image in $\uG$ contains regular semisimple elements.  The proposition again follows from Lemma~\ref{Ellers-Gordeev}.
\end{proof}

From this, we easily deduce Proposition~\ref{sln-big}.

\begin{proof}[Proof of Proposition~\ref{sln-big}]
Because $n$ and $q-1$ are relatively prime, $\PSL_n(q) = \SL_n(q)$, so
$$4 = \iw(\SL_n(q)) = \iw(\PSL_n(q)) \le \ccn(\PSL_n(q))\le 4.$$
\end{proof}

\section{Alternating groups and sporadic groups}

\begin{prop}\label{an-4}
For $n\ge 5$ we have $\ccn(\AAA_n)\le 4$.
\end{prop}

\begin{proof}
By \cite[p.~1]{AH}, we have $\ecn(\AAA_n)\le 5$ for $n\le 9$, so by Lemma~\ref{ecn}, we may assume $n\ge 10$.

If $p$ is a prime, the permutation representation on $\P^1(\F_p)$ embeds $\PSL_2(p)$ in $\AAA_{p+1}$.
Every element of order $p$ maps to a $p$-cycle in $\AAA_{p+1}$ and therefore to a $p$-cycle in $\AAA_n$ for all $n\ge p+1$.
By Lemma~\ref{L2}, if $p\equiv 1\pmod4$, then there exist elements
$$s\in S_1(\PSL_2(p))S_2(\PSL_2(p)),\ t\in S_3(\PSL_2(p))S_4(\PSL_2(p))$$
of order $p$, so every element in $\AAA_n$ which is a product of two $p$-cycles lies in $S_1(\AAA_n)\cdots S_4(\AAA_n)$.

By a theorem of Bertrand \cite{B}, if $\lfloor \frac{3n}4\rfloor \le p \le n$, then every element of $\AAA_n$ is a product of two $p$-cycles.
Therefore, $\ecn(\AAA_n)\le 4$ whenever there exists a prime which is $1$ (mod $4$) in $[\lfloor 3n/4\rfloor, n-1]$.   Applying this for $p=13$ and $p=17$, we get
the desired inequality for $n$ in  $\{14,15,16,17\}$ and $\{18,19,20,21,22,23\}$ respectively.  The following table gives
primes which together cover all values of $n\in [30, 1.3\cdot 10^{10}]$.

\begin{center}
\begin{tabular}{|l|l|l|l|l|}
\hline
29&37&41&53&61\\
73&97&113&149&197\\
257&337&449&593&773\\
1021&1361&1801&2393&3181\\
4241&5653&7537&10037&13381\\
17837&23773&31657&42209&56269\\
75017&99989&133277&177677&236897\\
315857&421133&561461&748613&998117\\
1330789&1774373&2365829&3154433&4205909\\
5607853&7477121&9969457&13292593&17723449\\
23631253&31508329&42011093&56014789&74686357\\
99581809&132775693&177034217&236045497&314727293\\
419636389&559515161&746020213&994693597& \\
\hline
\end{tabular}
\end{center}

By a theorem of Ramar\'e and Rumely \cite{RR}, if $n\ge 10^{10}$, then the sum of $\log p$ over all $p<n$ which are $1$ (mod $4$) lies between $.495 n$ and $.505 n$.
It follows that there is at least one such prime in the interval $[n,1.1n]$ for all $n>10^{10}$, and this is enough to show $\ccn(\AAA_n) \le 4$ for all $n>10^{10}$.

If $q$ is any odd prime power, then for $C_1$ and $C_2$ non-trivial conjugacy classes in $\PSL_2(q)$, there is an element of order $\frac{q+1}{2}$ in $C_1 C_2$.
This maps into a product of two disjoint $\frac{q+1}2$ cycles.
A theorem of Xu \cite{X} asserts that if $n-1\le r+s \le n$, then every element of $\AAA_n$ is a product of two elements, each consisting of two disjoint cycles of length $r$ and $s$.
Applying this for $q=9$, $q=11$, $q=23$, $q=25$, and $q=27$ we get $\ccn(\AAA_n)\le 4$ for $n$ in $\{10,11\}$, $\{12,13\}$, $\{24, 25\}$, $\{26,27\}$, and $\{28,29\}$ respectively,
so all cases are covered.
\end{proof}

To deal with $\AAA_n$ for large $n$, it is useful to have the following lemma:

\begin{lem}
\label{construct}
Let $p$ be a prime, $k$ and $n$ positive integers, $n \geq kp$, and $u_1,\ldots,u_k,v_1,\ldots,v_k\in \{1,\ldots,n\}$
such that $u_i\neq u_j$, $v_i\neq v_j$, and $u_i\neq v_j$ when $i\neq j$.  Then there exists an element $\sigma\in\SSS_n$ such that $\sigma(u_i)=v_i$ for $1\le i\le k$ and
$\sigma^p$ is the identity.
\end{lem}

\begin{proof}
For each $i$ such that $u_i=v_i$, let $X_i = \{u_i\}$, and for each $i$ such that $u_i\neq v_i$, let $X_i$ be a $p$-element set containing $u_i$ and $v_i$
and such that the $X_i$ are disjoint from one another.  We define $\sigma$ to be the identity on $\{1,\ldots,n\}\setminus \bigcup X_i$ and to cyclically permute the elements of each $X_i$ such that $\sigma(u_i) = v_i$.
\end{proof}

We can now prove Proposition~\ref{alt}.

\begin{proof}[Proof of Proposition~\ref{alt}]
First we note by \cite[Theorem 1.2]{TZ} that $\iw(\AAA_n) \geq 3$ when $n \geq 15$: indeed, not every element in $\AAA_n$ with
$n \geq 15$ is a product of two involutions. Hence it suffices to show that $\ccn(\AAA_n) \leq 3$ when $n$ is large enough.

We say a characteristic collection $S$ if of \emph{type} $r$ if $r$ is the smallest prime such that $S(\AAA_5)$ contains an element of order $r$.  If $S$ is ample it must be of type $2$, type $3$, or type $5$.  For any positive integer $n$ and $0\le m\le \lfloor n/5\rfloor$, the embedding $\AAA_5^m< \AAA_n$ guarantees that $S(\AAA_n)$ contains an element of cycle type $1^{5-4m}2^{2m}$
if $S$ is of type $2$, $1^{5-3m}3^m$ if $S$ is of type $3$, and $1^{5-5m}5^m$, if $S$ is of type $5$.

For $p\equiv\pm 1\pmod{10}$, $\AAA_5<\PSL_2(p)<\AAA_{p+1}$, and every element of $\AAA_5$ of order $r$, when regarded as an element $x$ of $\AAA_{p+1}$, has at most $2$ fixed points and otherwise consists entirely of $r$-cycles.  For $p=59$,  $x$ is therefore a permutation of the form $r^{60/r}$.   If $S$ is of type $2$, the embedding $\AAA_{60}^i\times \AAA_{n-60i} < \AAA_n$ gives
$S(\AAA_n)$  elements of cycle type $1^{5-4m}2^{2m}$
whenever $m$ belongs to
$$I_i = \biggl[15i,15i+\bigl\lfloor\frac{n-60i}5\bigr\rfloor\biggr].$$
When $75+60k\le n < 75+60(k+1)$, the consecutive intervals in the sequence $I_0,\ldots,I_k$ overlap, so for $0 \le m\le 15k$, all elements of cycle type $1^{5-4m}2^{2m}$ belong to $S(\AAA_n)$.
This means that every even permutation of order $2$ with at least $B_2 = 75$ fixed points lies in $S(\AAA_n)$.  By the same reasoning,  for $B_3=100$ and $B_5=60$, all elements of order $3$ and $5$ in $\AAA_n$  with at least $B_3$ and $B_5$ fixed points respectively lie in $S(\AAA_n)$ if $S$ is respectively of type $3$ or $5$.

By \cite[Theorem 1.2(ii)]{LS1}, if $x$ is an element of $\SSS_n$ of prime order $r$ and a bounded number of fixed points, then
for every irreducible character $\chi$ of $\SSS_n$, $|\chi(x)| \le \chi(1)^{1/r+o(1)}$.  Therefore, if $n$ is sufficiently large and
$x_1,x_2,x_3$ are three such elements in $\AAA_n$, of orders $r_1,r_2,r_3$ respectively,
and
$$\frac{1}{r_1}+\frac{1}{r_2}+\frac{1}{r_3}<1,$$
then the product of the $\SSS_n$-conjugacy classes of the $x_i$ cover $\AAA_n\smallsetminus \{1\}$. (Indeed,
note that we have $1/r_1+1/r_2+1/r_3 \leq 41/42$ in such a case, and so, by \cite[Theorem 2.6]{LiSh2},
$$\sum_{\chi \in \Irr(\SSS_n),~\chi(1) > 1}\frac{|\chi(x_1)\chi(x_2)\chi(x_3)|}{\chi(1)} <
    \sum_{\chi \in \Irr(\SSS_n),~\chi(1) > 1}\frac{1}{\chi(1)^{1/42}} \to 0$$
when $n \to \infty$.)
We may assume without loss of generality that $r_1\le r_2\le r_3$, so we need only consider the cases $(r_1,r_2)=(2,2)$, $(r_1,r_2)=(2,3)$, and $r_1=r_2=r_3 = 3$.
We proceed by case analysis.

\vskip 0.2cm\noindent{\bf Case} $(2,2,r_3)$.
For any $k \geq 1$, let
$$\sigma_k = (1,2)\cdots(2k-1,2k)\in \SSS_{2k},~\tau_k=(2,3)\cdots(2k,2k+1)\in\SSS_{2k+1}.$$
Thus, $\sigma_k \tau_k$ is a $(2k+1)$-cycle, while $\sigma_k \tau_{k-1}$ is a $2k$-cycle. Applying this to any cycle, we see that
every element in $\SSS_n$ can be written as a product of two involutions.
Regarding $\SSS_{n-2}$ as the stabilizer in  $\AAA_n$ of $\{n-1,n\}$, we see that every element of $\AAA_n$
which fixes $n-1$ and $n$ can be written as a product of two involutions in $\AAA_n$.  More generally, every element of $\AAA_n$ with at least $2+B_2$ fixed points can be written as a product of even involutions, each of which has at least $B_2$ fixed points and hence
belongs to $S(\AAA_n)$.

We claim that if $n$ is sufficiently large, for every $\rho\in \AAA_n$ and $p\in\{2,3,5\}$ there exists $\pi\in \AAA_n$ of order dividing $p$ with at least $B_p$ fixed points such that $\pi \rho$ has at least $2+B_2$ fixed points.
We choose a sequence $x_1,\ldots,x_{2+B_2}$ of distinct elements of $\{1,\ldots,n\}$ such that $\rho(x_i)\neq x_j$ for $i\neq  j$ and then a sequence $y_1,\ldots,y_{B_p}$ of distinct elements of $\{1,\ldots,n\}\smallsetminus \{\rho(x_1),\ldots,\rho(x_{2+B_2})\}$.
By Lemma~\ref{construct}, if $n$ is sufficiently large, there exists an even permutation $\pi$  of order $1$ or $p$ which fixes each $y_i$ and maps each $\rho(x_j)$ to $x_j$. By the above discussion, $\pi\rho \in S(\AAA_n)S(\AAA_n)$ and
$\pi^{-1} \in S(\AAA_n)$, hence $\rho \in S(\AAA_n)S(\AAA_n)S(\AAA_n)$.

\vskip 0.2cm\noindent{\bf Case} $(2,3,r_3)$.  Let
\begin{align*}
\sigma_k &= (1,2)(4,5)\cdots(3k-2,3k-1),\\
\sigma'_k &= (3,4)(6,7)\cdots(3k,3k+1),\\
\tau_k &= (2,3,4)(5,6,7)\cdots(3k-1,3k,3k+1),\\
\tau'_k &= (1,2,3)(4,5,6)\cdots(3k-2,3k-1,3k).
\end{align*}
Then $\sigma_k \tau_k$ is a $(3k+1)$-cycle, $\sigma_{k+1}\tau_k$ is a $(3k+2)$-cycle, and $\sigma'_k \tau'_{k+1}$ is a $(3k+3)$-cycle.
Thus, every element of $\SSS_n$ is a product of an involution $\sigma$ and an element $\tau$ of order $1$ or $3$.  Since any such $\tau$ lies in $\AAA_n$,
the same statement holds in $\AAA_n$, and if the product has at least $\max(B_2,B_3)$ fixed points, the same can be assumed of $\sigma$ and $\tau$.
Applying Lemma~\ref{construct} as before, we can guarantee for each $\rho\in \AAA_n$ the existence of an element
$\pi\in \AAA_n$ of order dividing $r_3$ with at least $B_{r_3}$ fixed points such that $\pi \rho$ has at least $\max(B_2,B_3)$ fixed points,
and then finish as in the previous case.

\vskip 0.2cm\noindent{\bf Case} $(3,3,3)$.  It suffices to prove that every element of $\AAA_n$ can be written as a product of two elements of order dividing $3$.  Let
\begin{align*}
\sigma_{k,l} &= (1,2,3)(5,6,7)\cdots(4k-7,4k-6,4k-5)(4k-3,4k-2,4k-1)\\
		&\qquad\qquad\qquad(4k,4k+1,4k+2)(4k+4,4k+5,4k+6)\cdots (4l,4l+1,4l+2),\\
\tau_{k,l} &= (3,4,5)(7,8,9)\cdots(4k-5,4k-4,4k-3)(4k-1,4k,4k+1)\\
		&\qquad\qquad\qquad(4k+2,4k+3,4k+4)(4k+6,4k+7,4k+8)\cdots (4l+2,4l+3,4l+4),
\end{align*}
where the first line in each expression is omitted if $k=0$ and the second line is omitted if $l=k-1$.
Then $\sigma_{k,k-1}\tau_{k-1,k-2}$ is a $(4k-1)$-cycle, and $\sigma_{k,k-1}\tau_{k,k-1}$ is a $(4k+1)$-cycle,
so all odd cycles can be written as a product of two elements of order dividing $3$.
For $l\ge k\ge 0$ and $k+l\ge 1$, $\sigma_{k,l}\tau_{k,l}$ is a disjoint product of a $4k$-cycle and a $(4l+4-4k)$-cycle; $\sigma_{k,l}\tau_{k-1,l}$ is a disjoint product of
a $(4k-2)$-cycle and a $(4l+6-4k)$-cycle; $\sigma_{k,l}\tau_{k,l-1}$ is a disjoint product of a $4k$-cycle and a $(4l+2-4k)$-cycle; and
$\sigma_{k,l}\tau_{k-1,l-1}$ is a disjoint product of a $(4k-2)$-cycle and a $(4l+4-4k)$-cycle. Thus, all possible permutations that are products of two disjoint even-length cycles can
be written as a product of two permutations of order dividing $3$.
\end{proof}

\begin{prop}\label{spor}
For all sporadic finite simple groups we have $\ccn(G) \le 4$.
\end{prop}

\begin{proof}
By a theorem of Zisser \cite{Z}, we have $\ecn(G)\le 5$ for all sporadic groups except $\Fi_{22}$ and $\Fi_{23}$, so it suffices to consider these two cases.
It is known \cite[pp.~74,~163]{Atlas} that there are inclusions $\PSL_2(25) < {}^2F_4(2)' < \Fi_{22}$.
By Lemma~\ref{L2}, if $S_1$ and $S_2$ are ample characteristic collections,
then $S_1(\PSL_2(25))S_2(\PSL_2(25))$ has an element of order $13$, and a machine computation shows that the product of any two conjugacy classes of elements of order $13$
in $\Fi_{22}$ contains $\Fi_{22}\smallsetminus\{1\}$.  It is also known \cite[p.~177]{Atlas} that $\PSL_2(17) < \Fi_{23}$.  By Lemma~\ref{L2}, the product $S_1(\PSL_2(17))S_2(\PSL_2(17))$
contains an element of order $17$; a machine computation shows that the square of the unique conjugacy class of order $17$ in $\Fi_{23}$ is the whole group.
\end{proof}

Together with Theorems \ref{Classical} and \ref{Exceptional}, Propositions \ref{an-4} and \ref{spor} complete the proof
of Theorem A. Theorem D follows from results \ref{an-4}, \ref{spor}, \ref{big-classical} and \ref{big-exceptional}.

\end{document}